\documentclass[11pt]{amsart}
\usepackage[utf8]{inputenc}
\usepackage[utf8]{inputenc}
\usepackage[english]{babel}
\usepackage[all]{xy}
\usepackage[margin=0.85in]{geometry}

\usepackage{amsmath, amsthm, amsfonts, amssymb}
\usepackage{color}
\usepackage{comment}
\usepackage{graphicx}

\usepackage{mathrsfs}
\usepackage{graphicx}

\newtheorem{thm}{Theorem}[section]

\newtheorem{cor}[thm]{Corollary}
\newtheorem{lem}[thm]{Lemma}
\newtheorem{prop}[thm]{Proposition}

\theoremstyle{definition}
\newtheorem{defn}[thm]{Definition}
\theoremstyle{remark}
\newtheorem{rem}[thm]{Remark}
\newcommand{\dblq}{{/\!/}}

\def\cM{\mathcal{M}}

\def\cO{\mathcal{O}}
\def\bP{\mathbb{P}}

\def\cQ{\mathcal{Q}}

\def\cX{\mathcal{X}}
\def\cC{\mathcal{C}}
\def\bZ{\mathbb{Z}}

\def\FF{{\textbf F}}
\def\LL{{\textbf L}}
\def\mm{\overline{\mathcal{M}}}

\def\wt{\widetilde}

\DeclareMathOperator{\CH}{\mathrm{CH}}
\DeclareMathOperator{\Gr}{Gr}

\DeclareMathOperator{\Hom}{Hom}
\DeclareMathOperator{\id}{id}

\DeclareMathOperator{\spine}{sp}

\DeclareMathOperator{\tot}{tot}
\DeclareMathOperator{\Tev}{Tev}
\DeclareMathOperator{\vir}{vir}

\begin{document}
\title{Linear series on general curves with prescribed incidence conditions}
\author[G. Farkas]{Gavril Farkas}

\address{Humboldt-Universit\"at zu Berlin, Institut f\"ur Mathematik,  Unter den Linden 6
\hfill \newline\texttt{}
 \indent 10099 Berlin, Germany} \email{{\tt farkas@math.hu-berlin.de}}

\author[C. Lian]{Carl Lian}

\address{Humboldt-Universit\"at zu Berlin, Institut f\"ur Mathematik,  Unter den Linden 6
\hfill \newline\texttt{}
 \indent 10099 Berlin, Germany} \email{{\tt liancarl@hu-berlin.de}}

\maketitle

\begin{abstract}
Using degeneration and Schubert calculus, we consider the problem of computing the number of linear series of given degree $d$ and dimension $r$ on a general curve of genus $g$ satisfying prescribed incidence conditions at $n$ points. We determine these numbers completely for linear series of arbitrary dimension when $d$ is sufficiently large, and for all  $d$ when either $r=1$ or $n=r+2$. Our formulas generalize and give new proofs of recent results of Tevelev and of Cela-Pandharipande-Schmitt.
\end{abstract}

\section{Introduction}\label{intro}

Having fixed positive integers $r$ and $s$ and setting $g=rs+s$ and $d=rs+r$, in a celebrated paper  \cite{castelnuovo}, Castelnuovo computed the number of linear series  of type $g^r_d$ on a general curve $C$ of genus $g$. By degeneration to a $g$-nodal rational curve, he argued that this number equals the degree of the Grassmannian $\Gr(r+1,d+1)$ in its Pl\"ucker embedding, that is,
$$g!\cdot\frac{1!\cdot 2! \cdot \cdots \cdot r!}{s!\cdot(s+1)! \cdot \cdots \cdot (s+r)!}.
$$
A rigorous modern presentation of Castelnuovo's argument\footnote{The fact that Castelnuovo provided  a plausibility argument rather than a complete proof has been immediately recognized. We quote from the Zentralblatt MATH review \cite{Lo} of \cite{castelnuovo}: \emph{Das Resultat, welches Herr Castelnuovo bekommen hat, gibt mit grosser Wahrscheinlichkeit den wahren Wert,  weil sein Forderungssatz ... sehr leicht angenommen werden kann; doch k\"onnen wir unseren Wunsch nicht unterdr\"ucken, die obige Aufgabe auf einspruchsfreie Weise aufgel\"ost zu sehen.}} was first carried out by Griffiths and Harris \cite{GH80}.
More generally, the theory of limit linear series developed by Eisenbud and Harris \cite{eh_genus0,eh_lls} allows one to compute the number of linear series on a general curve with ramification conditions imposed at fixed marked points, see also \cite{osserman} for a more recent treatment.

\vskip 4pt

Motivated by two recent papers of Tevelev \cite{tevelev} and Cela-Pandharipande-Schmitt \cite{cps}, we consider a variant of this problem, where we impose incidence conditions on the corresponding maps to projective spaces. Let $[C,x_1,\ldots,x_n]\in\cM_{g,n}$ be a general $n$-pointed complex curve of genus $g$. We denote by $G^r_d(C)$ the variety of linear systems $\ell=(L,V)$ of type $g^r_d$ on $C$. A general $\ell\in G^r_d(C)$ corresponds to a regular map $\phi_{\ell}\colon C\rightarrow \mathbb P^r$. Evaluation at the points $x_1, \ldots, x_n$ induces a rational map

\begin{equation}\label{eq:evaluation}
\mathrm{ev}_{(x_1, \ldots, x_n)}\colon G^r_d(C)\dashrightarrow \bigl(\mathbb P^r)^n\dblq PGL(r+1)=:P_r^n,
\end{equation}
to the moduli spaces of $n$ points in $\bP^r$.\footnote{The GIT quotient $\bigl(\mathbb P^r)^n\dblq PGL(r+1)$ depends on a choice of linearization, but our main point of study, the degree of $\mathrm{ev}_{(x_1, \ldots, x_n)}$, is independent of this choice.}
We study the degree $L_{g,r,d}$ of the map $\mathrm{ev}_{(x_1, \ldots, x_n)}$ in the case when this map is generically finite and both spaces have non-negative dimension.  Since $G^r_d(C)$ is a smooth variety of dimension $\rho(g,r,d)=g-(r+1)(g-d+r)$, whereas $\mbox{dim}(P_r^n)=rn-r^2-2r$ as long as $n\geq r+2$, one expects $\mathrm{ev}_{(x_1, \ldots, x_n)}$ to be generically finite precisely when

\begin{equation}\label{eq:n}
n=\frac{dr+d+r-rg}{r}.
\end{equation}

Equivalently, $L_{g,r,d}$ may be understood as the degree of the morphism
\begin{equation*}\label{eq:evaluation_maps}
\tau:\cM_{g,n}(\bP^r,d)\to\cM_{g,n}\times(\bP^r)^n
\end{equation*}
where $\cM_{g,n}(\bP^r,d)$ is the moduli space of degree $d$ maps $f:C\to\bP^r$, with smooth domain and distinct marked points $x_1,\ldots,x_n\in C$, and the map $\tau$ remembers the pointed domain and the images of the $x_i$ under $f$. Again, the map $\tau$ is expected to be generically finite exactly when \eqref{eq:n} holds, and Brill-Noether theory guarantees that this is indeed the case as long as $d>0$.

If $y_1,\ldots,y_n\in\bP^r$ are general points, $L_{g,r,d}$ counts the number of morphisms $f\colon C\to\bP^r$ of degree $d$ satisfying $f(x_i)=y_i$ for $i=1, \ldots, n$. When the points $y_i$ are considered up to projective equivalence, these incidence conditions are intrinsic to $\ell$.  For large $d$, it turns out there is a very simple formula for this degree:

\vskip 4pt

\begin{thm}\label{thm_d_large}
Suppose $d\ge rg+r$, or equivalently $n\ge d+2$. Then
$$L_{g,r,d}=(r+1)^{g}.$$
\end{thm}

\vskip 3pt

We remark that the hypothesis $n\ge d+2$ is automatically satisfied whenever $g\le1$. Indeed, if instead $n\le d+1$ and $g\le1$, then $d+1\ge n=d+1+\frac{d}{r}-g\ge d+\frac{d}{r}$, hence $n\le d+1\le r+1$, a contradiction. On the other hand, we will also see that the inequality $d\ge rg+r$ is sharp in the sense that $L_{g,r,d}=(r+1)^{g}-(d+1) $ when $d=rg$, see Remark \ref{ineq_sharp}.

\vskip 3pt

When $r=1$, the special case $d=g+1$ was studied under the guise of \textit{scattering amplitudes} by Tevelev \cite{tevelev}, who found the strikingly simple formula $L_{g,1,g+1}=2^g$. This raised the possibility, confirmed by Theorem \ref{thm_d_large}, that in the range when $d$ is relatively large, the degree $L_{g,r,d}$ has a simple expression. Using Hurwitz space techniques, Cela-Pandharipande-Schmitt \cite{cps} obtained general formulas for $L_{g,1,d}$, which they called \textit{Tevelev degrees}; in particular, when $d\ge g+1$, they found again $L_{g,1,d}=2^g$.

\vskip 3pt

The result of Theorem \ref{thm_d_large} can be compared to  a certain virtual count of Bertram-Daskalopoulous-Wentworth \cite[Theorem 2.9]{bdw} in the range $d>2g-2$, pre-dating the theory of virtual fundamental classes on moduli spaces of stable maps.

When either $r=1$ or $n=r+2$, or under the hypotheses of Theorem \ref{thm_d_large}, we obtain a more general formula for $L_{g,r,d}$ in terms of Schubert calculus. For a positive integer $a$, we recall the notation $\sigma_a$ for the class of the special Schubert cycle of codimension $a$ consisting of those $(r+1)$-planes $V\in \Gr(r+1,d+1)$ meeting a fixed subspace $W\subseteq \mathbb C^{d+1}$ of dimension $d-a$. We also recall that $\sigma_{1^r}$ denotes the class of the special Schubert cycle of codimension $r$ consisting of those $(r+1)$-planes $V\in \Gr(r+1,d+1)$ whose intersection with a fixed codimension 2 linear subspace $U\subseteq \mathbb C^{d+1}$ has dimension at least $r$. Our main result is as follows:

\begin{thm}\label{thm_r_or_g_small}
Suppose that either:
\begin{itemize}
\item $d\ge rg+r$, (i.e., the same hypothesis as in Theorem \ref{thm_d_large}),
\item $d=r+\frac{rg}{r+1}$ (in which case $n=r+2$), or
\item $r=1$.
\end{itemize}
In each of these cases,
\begin{equation*}
L_{g,r,d}=\int_{\Gr(r+1,d+1)}\sigma_{1^r}^{g}\cdot\left[ \sum_{\alpha_0+\cdots+\alpha_{r}=(r+1)(d-r)-rg}\left(\prod_{i=0}^{r}\sigma_{\alpha_i}\right)\right].
\end{equation*}
\end{thm}

In particular, comparing Theorem \ref{thm_d_large} with Theorem \ref{thm_r_or_g_small} when $d\ge rg+r$ yields a non-trivial combinatorial identity\footnote{ A combinatorial proof of this identity has been given by Gillespie-Reimer-Berg \cite{grb} after our paper appeared on arXiv, see \S\ref{tableaux} for a discussion.}.

In the second case, in which $d$ is as small as possible, we have $(r+1)(d-r)=gr$, so the second term is interpreted to be 1, and Theorem \ref{thm_r_or_g_small} recovers Castelnuovo's formula for $s=g/(r+1)$. On the other hand, when $gr>(r+1)(d-r)$ (equivalently, $n<r+2$), the summation is interpreted to be zero, so that $L_{g,r,d}=0$. Indeed, this corresponds to the case $\dim G^{r}_{d}(C)=\dim(P_r^n)<0$, in which we find no such morphisms $f\colon C\to\bP^r$.

The case of intermediate $d$ when $r>1$ is the most subtle, and will be addressed in later work.

\vskip 4pt

For $r=1$, Theorem \ref{thm_r_or_g_small}, via Giambelli's formula, yields the following explicit formulas for $L_{g,1,d}$, the last of which agrees with the results of \cite[Theorem 6]{cps}, see Proposition \ref{prop:explicit} for details.

\begin{align*}
L_{g,1,d}&=\sum_{\alpha_0+\alpha_1=2d-2-g}\int_{\Gr(2,d+1)} \sigma_1^g\cdot \sigma_{\alpha_0}\cdot \sigma_{\alpha_1}\\
&=
\sum_{i=0}^{\bigl \lfloor \frac{2d-g-2}{2}\bigr \rfloor} \frac{(2d-g-2i-1)^2}{g+1} {g+1\choose d-i}\\
&=2^g-2\sum_{i=0}^{g-d-1}\binom{g}{i}+(g-d-1)\binom{g}{g-d}+(d-g-1)\binom{g}{g-d+1}.
\end{align*}
Here, we adopt the convention that $\binom{g}{j}=0$ when $j<0$, so in particular we have again that $L_{g,1,d}=2^g$ when $d\ge g+1$.
\vskip 3pt

To prove Theorems \ref{thm_d_large} and \ref{thm_r_or_g_small}, we proceed via a standard degeneration to a flag curve consisting of a rational spine and $g+1$ tails, one of which being rational and on which all the marked points specialize, the remaining $g$ tails being elliptic curves. We  reduce to a concrete problem in genus zero in \S\ref{degeneration}, and  then handle this problem via Schubert calculus in \S\ref{main_computation}; care needs to be taken to avoid degenerate solutions, particularly excess contributions from constant maps $f\colon \bP^1\to\bP^r$ obtained from linear series with base points at some of of the points $x_i$. Because such excess contributions in our setup persist when the hypotheses of Theorem \ref{thm_r_or_g_small} are not satisfied (see Remark \ref{const_higherrank}), the general computation of $L_{g,r,d}$ remains open.

\vskip 4pt

Our method in the case $r=1$ also allows us to re-compute the more general counts of Cela-Pandharipande-Schmitt \cite{cps}, where some points of the source curve are constrained to have the same image. If $r=1$ and $1\le k\le d,n$, let $L'_{g,d,k}$ be the number of morphisms $f\colon C\to\bP^1$ as before, but where we take $y_1=y_2=\cdots=y_k$, and the $y_i$ otherwise general. (Note that our indexing differs from that of \cite{cps}, where $d$ is written as $g+1+\ell$ for some $\ell\in\bZ$ and $k$ is called $r$, whereas we have reserved the variable $r$ to denote the dimension of the target projective space.) We find:

\begin{thm}\label{thm_cps}
\begin{equation*}
L'_{g,d,k}=\int_{\Gr(2,d+1)}\sigma_1^g\sigma_{k-1}\cdot\left[\sum_{i+j=2(d-1)-g-(k-1)}\sigma_{i}\sigma_j\right]-\int_{\Gr(2,d)}\sigma_1^g\sigma_{k-2}\cdot\left[\sum_{i+j=2(d-2)-g-(k-2)}\sigma_{i}\sigma_j\right].
\end{equation*}
\end{thm}

The second term is taken to be zero when $k=1$. Note that $L'_{g,d,1}=L_{g,1,d}$, so Theorem \ref{thm_cps} agrees with Theorem \ref{thm_r_or_g_small} in the case $r=1$. From here, the formulas of \cite[Theorem 6]{cps} can be recovered by recursion, see Corollary \ref{recursion}. We sketch the proof of Theorem \ref{thm_cps} in \S\ref{tev_computation}; a more general statement with detailed proofs appears is given in \cite[\S 6]{celalian}.

Finally, we remark that the degeneration technique also allows one to impose ramification conditions at additional fixed points $p_1,\ldots,p_m\in C$, see \S\ref{ram_conditions}.

\vskip 4pt

\noindent {\bf Relation to other work.}  We discuss  results related to this circle of ideas that appeared after our paper was published  on arXiv.
The count of Theorem \ref{thm_d_large} agrees with a virtual count of maps $C\to\bP^r$ in Gromov-Witten theory as computed by Buch-Pandharipande \cite{buchpand}, the so-called \textit{virtual Tevelev degrees} of $\bP^r$.  Let
$\tau\colon \mm_{g,n}(\bP^r,d)\to\mm_{g,n}\times(\bP^r)^n$
be the map remembering $[(C,x_1,\ldots,x_n)]$ and the points $y_i=f(x_i)$. Then, under assumption \eqref{eq:n}, we have  \cite[\S 1.3]{buchpand}
\begin{equation*}
\tau_{*}([\mm_{g,n}(\bP^r,d)]^{\vir})=(r+1)^g\cdot[\mm_{g,n}\times(\bP^r)^n].
\end{equation*}
for \emph{all} $d$. When $d<rg+r$, the virtual count includes excess contributions we wish to exclude in our counts $L_{g,r,d}$.

More generally, when $\bP^r$ is replaced by an arbitrary target variety $X$, the corresponding virtual degrees are expressed in terms of the quantum cohomology of $X$, see \cite[Theorem 1.3]{buchpand}. It is expected that for all Fano varieties $X$, the virtual count of maps of sufficiently large degree is enumerative, as in Theorem \ref{thm_d_large}, see \cite{lianpand} for partial results in this direction.
\vskip 4pt

{\small{\noindent {{\bf{Acknowledgments:}}} This problem was brought to our attention by R. Pandharipande. We thank him as well as A. Cela and J. Schmitt for interesting discussions related to this circle of ideas. We also thank L. Manivel and M. Gillespie for pointing out an error in \S\ref{tableaux} in an earlier version and communicating alternate proofs of Proposition 4.1, and B. Gunby for help with combinatorial manipulation. Finally, we thank the anonymous referee for numerous improvements.

Farkas has been supported by the DFG Grant \emph{Syzygien und Moduli} and by the ERC Advanced Grant SYZYGY. This project has received funding from the European Research Council (ERC) under the European Union Horizon 2020 research and innovation program (grant agreement No. 834172).}
Lian was supported by an NSF Postdoctoral Fellowship, grant DMS-2001976.}

\section{Reduction to genus zero}\label{degeneration}

In this section, we reduce the enumerative problem to genus zero via a standard limit linear series degeneration, see for example \cite{osserman,lian_pencils}. We begin by recalling some notation, while assuming throughout some familiarity with basics of the theory of limit linear series \cite{eh_lls}.

\vskip 3pt

We recall the usual notation for Schubert cycles in the Grassmannian $\Gr(r+1,d+1):=\Gr(r+1,V)$, where $V$ is an $(d+1)$-dimensional vector space, following \cite[\S 4]{eh3264}. For a nonincreasing sequence $\mu:=(\mu_0\geq \mu_1\geq \cdots \geq \mu_r)$ and a flag $\FF: V=V_{d+1}\supset V_d\supset \ldots \supset V_1\supset V_0=0$, we introduce the Schubert cycle
$$\Sigma_{\mu}=\Sigma_{\mu}(\FF):=\Bigl\{\Lambda \in \Gr(r+1,d+1): \mbox{dim } \bigl(\Lambda \cap V_{d-r+1+i-\mu_i}\bigr)\geq i+1, \ \mbox{ for } i=0,\ldots, r\Bigr\}.$$
Note that $\mbox{codim}\bigl(\sigma_{\mu}, \Gr(r+1,d+1)\bigr)=|\mu|=\mu_0+\cdots+\mu_r$. If $\mu=(1,\ldots, 1, 0)=:1^r$, in projective terms  $\Sigma_{1^r}$ consists of $r$-dimensional subspaces  $L=\mathbb P(\Lambda)\subset \mathbb P(V)\cong \mathbb P^d$ intersecting a fixed codimension $2$ subspace along an $r$-dimensional locus. We set $\sigma_{\mu}:=[\Sigma_{\mu}]\in \CH^{|\mu|}\bigl(\Gr(r+1,g+1)\bigr)$.

\vskip 4pt

For a smooth curve $C$ and a linear series $\ell=(L,V)\in G^r_d(C)$, we denote by $$\alpha^{\ell}(p):=\bigl(0\leq \alpha_0^{\ell}(p)\leq \alpha_1^{\ell}(p)\leq \cdots \leq \alpha_r^{\ell}(p)\leq d-r\bigr)$$
the \emph{ramification sequence} at a point $p\in C$. Keeping with the tradition of \cite{GH80} or \cite{eh_lls}, we write ramification indices of linear series nondecreasingly, whereas indices indexing Schubert cycles are written nonincreasingly. We formalize this practice as follows:

\begin{defn}\label{schindex}
For any partition $\mu=(\mu_0\geq \mu_1\geq \cdots \geq \mu_r)$, denote by $\overline{\mu}$ the tuple of components of $\mu$ in reverse (increasing) order, that is, $\overline{\mu}=(\mu_r,\ldots,\mu_0)$.
\end{defn}

\vskip 3pt

We introduce the proper  stack of limit linear series of type $g^r_d$
$$\sigma \colon \widetilde{\mathcal{G}}^r_d\rightarrow \mathcal{M}_g^{\mathrm{ct}}$$
over the moduli space $\cM_g^{\mathrm{ct}}$ of curves of compact type.   For a curve $C$ of compact type, we denote by $\overline{G}^r_d(C)$ the variety of limit linear series on $C$. For pairwise distinct smooth points $p_1, \ldots, p_n\in C_{\mathrm{reg}}$ and Schubert indices  $\alpha^i=\bigl(0\leq \alpha_0^i\leq \cdots \leq \alpha_r^i\leq d-r\bigr)$, where $i=1, \ldots, n$, we set
$$\overline{G}^r_d\Bigl(C, (p_1, \alpha^1), \ldots, (p_n, \alpha^n)\Bigr):=\bigl\{\ell\in \overline{G}^r_d(C): \alpha^{\ell}(p_i)\geq \alpha^i, \mbox{ for } i=1, \ldots, n\bigr\},$$
viewed as a generalized degeneracy locus of expected dimension
\begin{equation}\label{adjBN}
\rho(g,r,d, \alpha^1, \ldots, \ldots, \alpha^n):=g-(r+1)(g-d+r)-\sum_{i=1}^n \sum_{j=0}^r \alpha_j^i.
\end{equation}

\vskip 4pt

We now consider a degeneration to the following flag curve of genus $g$, already considered in \cite{eh_lls}. Let $[C_0, x_1, \ldots, x_n]$ be the $n$-pointed genus $g$ curve of compact type consisting of a rational spine $R_{\spine}$ to which general elliptic tails $E_1,\ldots,E_g$ are attached at general points $p_1, \ldots, p_g\in R_{\spine}$ respectively, and a further rational component $R_0$, also attached at a general point $x_0$ of $R_{\spine}$. The marked points $x_1,\ldots,x_n$ specialize to general points of $R_0$.

\vskip 4pt

Let $\cC\rightarrow (B, b_0)$ be the versal deformation space of $[C_0, x_1, \ldots, x_n]$ and denote by $\tau_1, \ldots, \tau_n \colon  B\rightarrow \cC$ the sections corresponding to the $n$ marked points. We may assume that each point of $B$ parametrizes an $n$-pointed curve of genus $g$ of compact type. We further consider the induced moduli map
$B\rightarrow \cM_g^{\mathrm{ct}}$ forgetting the markings and let
$$\sigma_B\colon \widetilde{\mathcal{G}}^r_d/B:=\widetilde{\mathcal{G}}^r_d\times _{\cM_g^{\mathrm{ct}}} B\rightarrow B$$
be the corresponding family of limit linear series and consider the evaluation map
\begin{equation}\label{evmap}
\mathrm{ev}\colon \widetilde{\mathcal{G}}^r_d/B\dashrightarrow B\times P_r^n, \ \
\bigl(C_b, \ell \bigr)\mapsto \Bigl(b, \Bigl(\phi_{\ell}(\tau_1(b)), \ldots, \phi_{\ell}(\tau_n(b))\Bigr)\Bigr),
\end{equation}
where $\phi_{\ell}$ denotes the rational map to $\mathbb P^r$ corresponding to the aspect of the limit linear series $\ell$ on the component of $C_b$ on which all the marked points $\tau_1(b), \ldots, \tau_n(b)$ lie.

\vskip 4pt

Using \cite[Theorem 1.1]{EH2} it follows that $\widetilde{\mathcal{G}}^r_d/B$ is smooth of dimension $3g-3+n+\rho(g,r,d)$ over $B$ and every limit linear series on $C_0$ smooths to a linear series on a neighboring smooth curve. It follows that $\mbox{deg}(\mathrm{ev})=L_{d,g,r}$. We will determine this degree by looking at the scheme-theoretic fibre $\mathrm{ev}^{-1}(b_0, y_1, \ldots, y_n)$, where $y_1, \ldots, y_n$ are general points in $\bP^r$ considered up to projective equivalence. We will show in Lemma \ref{intersection_generic} that every point  $[C_0, \ell]\in \mathrm{ev}^{-1}(b_0, y_1, \ldots, y_n)$ corresponds to a limit linear series $\ell\in \overline{G}^r_d(C_0)$ which is base point free at each point $x_1, \ldots, x_n$. In particular, $\mathrm{ev}^{-1}(b_0, y_1, \ldots, y_n)$ is disjoint from the indeterminacy locus of the map $\mathrm{ev}$ defined in (\ref{evmap}).

\vskip 4pt

To that end, we wish to count limit linear series $\ell$ on $C_0$ of degree $d$ and rank $r$, subject to the condition that, after twisting down base points on the $R_0$-aspect, the points $x_1,\ldots,x_n$ have prescribed images in $\bP^r$. For a limit linear series $\ell$ on $C_0$, we denote by $\ell_{R_0}, \ell_{R_{\spine}}$ and $\ell_{E_i}$ its corresponding aspects. By the additivity of the Brill-Noether number for $\ell$ encoded in the very definition of a limit linear series, we have the following inequality
$$\rho(g,r,d)\geq \rho\Bigl(\ell_{R_0}, \alpha^{\ell_{R_0}}(x_0)\Bigr) + \rho\Bigl(\ell_{R_{\mathrm{sp}}}, \alpha^{\ell_{R_{\spine}}}(x_0), \alpha^{\ell_{R_{\mathrm{sp}}}}(p_1),\ldots, \alpha^{\ell_{R_{\spine}}}(p_g)\Bigr)+
\sum_{i=1}^g \rho\Bigl(\ell_{E_i}, \alpha^{\ell_{E_i}}(p_i)\Bigr).
$$

Since over the curve $[C_0, x_1, \ldots, x_n]$ the map $\mathrm{ev}$ evaluates the $R_0$-aspect of each limit linear series, it follows that
we must only consider the components of $\overline{G}^r_d(C_0)$ in which $\ell_{R_0}$ varies in a family of dimension $\rho(g,r, d)=\dim P_r^n$. This happens when the remaining aspects of $\ell$ satisfy $\rho\bigl(\ell_{E_i}, \alpha^{\ell_{E_i}}(p_i))=0$ for $i=1, \ldots, g$ and $\rho\Bigl(\ell_{R_{\spine}}, \alpha^{\ell_{R_{\spine}}}(x_0), \alpha^{\ell_{R_{\spine}}}(p_1),\ldots, \alpha^{\ell_{R_{\mathrm{sp}}}}(p_g)\Bigr)=0$.

This implies that on each elliptic tail $E_i$, the ramification sequence at the node $p_i$ must be equal to  $(d-r-1, \ldots, d-r-1,d-r)$. Indeed, we need $\alpha^{\ell_{E_i}}_r(p_i)=d-r$, or else
\begin{equation*}
\alpha^{\ell_{E_i}}_0(p_i)+\cdots+\alpha^{\ell_{E_i}}_r(p_i)\le r(d-r-1),
\end{equation*}
but also $\alpha^{\ell_{E_i}}_{r-1}(p_i)\le d-r-1$, or else $E_i$ would carry a linear series of rank 1 and degree 1. We therefore have a unique choice of the $E_i$-aspect, precisely $\ell_{E_i}=(d-r-1)p_i+\bigl|(r+1)p_i\bigr|$, for $i=1, \ldots, g$. By compatibility of the aspects of limit linear series, we find that $\alpha^{\ell_{R_{\spine}}}(p_i)=(0, 1, \ldots, 1)$ for $i=1, \ldots, g$, that is, $\ell_{R_{\spine}}$ has a simple cusp at each of the points $p_i, \ldots, p_g$.

From here, on $R_{\spine}$, the ramification sequence of $\ell$ at the point $x_0$ $$\left(\alpha^{\ell_{R_{\spine}}}_{0}(x_0),\ldots,\alpha^{\ell_{R_{\spine}}}_{r}(x_0)\right)$$
must satisfy the equality
$$\sum_{j=0}^{r}\alpha^{\ell_{R_{\spine}}}_{j}(x_0)=(r+1)(d-r)-rg,$$
whereas the $R_0$-aspect of $\ell$ satisfies
\begin{equation}\label{R0asp}
\sum_{j=0}^r \alpha_j^{\ell_{R_0}}(x_0)=rg.
\end{equation}

Let $\mu=(\mu_0\geq \cdots \geq \mu_r):=\overline{\alpha^{\ell_{R_{\spine}}}_{r}(x_0)}$, that is, we write the partition $\left(\alpha^{\ell_{R_{\spine}}}_{r}(x_0),\ldots,\alpha^{\ell_{R_{\spine}}}_{0}(x_0)\right)$, where the ramification indices are given in descending order, and let $\lambda$ be the complement of $\mu$ in $(d-r)^{r+1}$, that is, $\lambda_j=d-r-\mu_j$, for $j=0,1,\ldots,r$. Summarizing the discussion so far, for each limit linear series $\ell$ on $C_0$ contributing towards the degree of the map $\mathrm{ev}$ one has

\begin{equation}\label{R0ram}
\alpha^{\ell_{R_0}}(x_0)=\overline{\lambda}
\end{equation}.

The number of possible aspects $\ell_{R_{\spine}}$ on $R_{\spine}$ with ramification sequence $\overline{\mu}$ at $x_0$ and cusps at $p_1, \ldots, p_g$ is given by
\begin{equation*}
\beta_{\lambda}:=\int_{\Gr(r+1,d+1)}\sigma_{1^r}^g\cdot \sigma_{\mu}.
\end{equation*}

The transversality of the intersection follows from \cite{eh_genus0}, see also \cite{mtv}, that is, for a general choice of the points $p_1, \ldots, p_g$ and $x_0$, one has precisely $\beta_{\lambda}$ distinct linear series on $R_{\spine}$ with these property. Since $\sigma_{\lambda'}\cdot \sigma_{\mu}=0$, for any Schubert index $\lambda'\neq \lambda$ with $|\lambda'|=rg$, whereas $\sigma_{\lambda}\cdot \sigma_{\mu}=1$, we can write
\begin{equation}\label{sigma1_power}
\sigma_{1^r}^g=\sum_{|\lambda|=rg}\beta_{\lambda}\cdot \sigma_{\lambda}\in \CH^{g}\bigl(\Gr(r+1,d+1)\bigr).
\end{equation}

\begin{defn}\label{ls_count_lambda}
Given a partition $\lambda=\bigl(\lambda_0\geq \cdots \geq \lambda_r)$ with $|\lambda|=rg$ and general points $y_1, \ldots, y_n\in \bP^r$, we define $L_{g,r,d,\lambda}$ to be be the number of maps $f\colon \bP^1\to\bP^r$ of degree $d-\lambda_r$ sending $x_i$ to $y_i$ for $i=1, \ldots, n$ and with ramification sequence given by $\overline{\lambda}$ at $x_0$.
\end{defn}

Such maps are obtained by twisting the $R_0$-aspect of each limit linear series $\ell$ on $C_0$ by the order $\lambda_r$ of its base point $x_0$. Our degeneration shows:

\begin{prop}\label{degen_formula} For a general $n$-pointed curve $[C,x_1, \ldots, x_n]$ of genus $g$, the degree of the map $\mathrm{ev}_{(x_1, \ldots, x_n)}$ is given by the formula
\begin{equation*}
L_{g,r,d}=\sum_{|\lambda|=rg}\beta_{\lambda}L_{g,r,d,\lambda}.
\end{equation*}
\end{prop}
\begin{proof}
We have already explained that $L_{g,r,d}$ is the degree of the map $\mbox{ev}\colon \widetilde{\mathcal{G}}^r_d/B\dashrightarrow B\times P_r^n$. Having fixed general points $y_1, \ldots, y_n\in \mathbb P^r$, the fibre over $(b_0, y_1, \ldots, y_n)\in B\times P_r^n$ of the map $\mathrm{ev}$
is then \emph{scheme-theoretically} isomorphic to the variety of limit linear series $\ell\in \overline{G}^r_d(C_0)$, whose $R_0$-aspect maps the marked points $x_i$ to $y_i$ for $i=1, \ldots, n$. From the discussion above it follows that $\overline{G}^r_d\bigl(C_0)$ contains $\beta_{\lambda}$ components all isomorphic to  the variety $G^r_d\bigl(R_0, (x_0, \overline{\lambda})\bigr)$; the remaining components of $\overline{G}^r_d(C_0)$ do not contribute to the degree of $\mathrm{ev}$. Finally, observe that $L_{g,r,d,\lambda}$ is precisely the contribution to the degree of the map $\mathrm{ev}$ corresponding to the component $G^r_d\bigl(R_0, (x_0, \overline{\lambda})\bigr)$.
\end{proof}


\section{Counting linear series with assigned incidences on $\bP^1$}\label{main_computation}

Having reduced both Theorem \ref{thm_d_large} and \ref{thm_r_or_g_small} to a question on rational curves, we use Schubert calculus to complete their proofs.

Let us first sketch the argument. The set of maps $\bP^1\to\bP^r$ counted by the number $L_{g,r,d,\lambda}$ naturally sits inside the projective space $\bP^{(r+1)(d+1)-1}$ parametrizing morphisms $f=[f_0, \ldots, f_r]$ of degree $d$, as given by the intersection of the conditions
\begin{enumerate}
\item[(i)] $f(x_i)=y_i$ for $i=1, \ldots,n$,
\item[(ii)] $f$ has ramification at least $\overline{\lambda}$ at $x_0$.
\end{enumerate}
The conditions $f(x_i)=y_i$ cut out linear subspaces, while, upon summing over all $\lambda$ with the multiplicities $\beta_\lambda$, the ramification conditions at $x_0$ cut out an intersection of $g$ subvarieties of degree $r+1$. The expected degree of the intersection is therefore $(r+1)^g$, and we show in the proof of Theorem \ref{thm_d_large} that this intersection is indeed transverse when $d\ge rg+r$.

In general, however, the intersection described above has many excess components. Under the conditions of Theorem \ref{thm_r_or_g_small}, we remove these excess contributions by passing to a certain incidence correspondence dominating $\bP^{(r+1)(d+1)-1}$ to compute $L_{g,r,d}$.

\subsection{Proof of Theorem \ref{thm_d_large}}

For a complex polynomial $u=a_0+\cdots +a_dt^d$ we denote by $c(u)$ the column vector of its coefficients. Let $\bP^{(r+1)(d+1)-1}$ be the projective space parametrizing $(r+1)$-tuples $(f_0,\ldots,f_r)$ of polynomials of degree $d$ in one variable viewed as sections of $\cO_{\bP^1}(d)$, up to simultaneous scaling, and not all zero. When not all polynomials $f_i$ are zero and have no common zeroes, they define a map $f=[f_0, \ldots, f_r]$ of degree $d$ from $\bP^1$ to $\bP^r$.

We introduce the map
\begin{equation}\label{eq:pi}
\pi\colon \bP^{(r+1)(d+1)-1}\dashrightarrow \Gr(r+1,d+1)
\end{equation}
remembering the linear series spanned by $f_0, \ldots, f_r$, whenever they are linearly independent. The indeterminacy locus of this map is irreducible of codimension $d-r+1$, for an $(r+1)$-tuple of polynomials $(f_0, \ldots, f_r)$ lies in the indeterminacy locus of $\pi$ if an only if the $(r+1)\times (d+1)$-matrix of coefficients $\bigl(c(f_0), \ldots, c(f_r)\bigr)$ has rank at most $r$.

\vskip 4pt

For a Schubert variety $\Sigma_\lambda=\Sigma_{\lambda}(\FF)\subseteq \Gr(r+1,d+1)$ of codimension at most $rg$ in $\Gr(r+1,d+1)$, let $\wt{\Sigma}_{\lambda}:=(\pi_1)_*\bigl(\pi_2^*(\Sigma_{\lambda})\bigr)$ be the closure of its pullback under $\bP^{(r+1)(d+1)-1}$. Because the codimension of $\Sigma_{\lambda}$ is lower that that of the indeterminacy locus of $\pi$ (by our assumption $d-r+1>rg$), the cycle $\wt{\Sigma}_{\lambda}$ has the expected codimension of $|\lambda|=rg$ and defines a well-defined class $\wt{\sigma}_{\lambda}\in \CH^{rg}\bigl(\bP^{(r+1)(d+1)-1}\bigr)$. Using (\ref{sigma1_power}) we have the formula
\begin{equation*}
\wt{\sigma}_{1^r}^{g}=\sum_{|\lambda|=rg}\beta_{\lambda}\wt{\sigma}_{\lambda}.
\end{equation*}

\vskip 3pt

Recall that we have fixed $n$ general points $y_1, \ldots, y_n\in \bP^r$. The condition on maps $f\colon \bP^1\to\bP^r$ that $f(x_i)=y_i$ for $i=1, \ldots, n$ impose $nr$ linear conditions on the matrix of coefficients $\bigl(c(f_0), \ldots, c(f_r)\bigr)$. Observe that this condition is automatically satisfied for those $i$ for which $x_i$ is a base point of $f$. The points $y_1, \ldots, y_n$ having been chosen to be general, these linear conditions are independent. Since $(r+1)(d+1)-1-nr=rg$, the conditions $f(x_i)=y_i$ give rise to a linear subspace

$$\LL\cong\bP^{rg}\subseteq \bP^{(r+1)(d+1)-1}.$$

Now, let ${\Sigma}_{\lambda}(x_0)=G^r_d\bigl(\bP^1, (x_0, \overline{\lambda}\bigr))$ be the Schubert variety of $\Gr(r+1,d+1)$ parametrizing linear series  on $\bP^1$ with ramification sequence at least $\overline{\lambda}$ at $x_0$. We wish to intersect its pullback $\wt{\Sigma}_{\lambda}(x_0)$ with $\LL$ on $\bP^{(r+1)(d+1)-1}$. We call a point $[f_0,\ldots,f_r]$ in this intersection \textit{generic} if $\langle f_0,\ldots,f_r\rangle$ is a linear series of rank $r$ with ramification sequence exactly $\overline{\lambda}$, and which defines a (non-degenerate) morphism $f\colon \bP^1\to\bP^r$ after twisting down the base points at $x_0$ with $f(x_i)=y_i$ (in particular, $\langle f_0,\ldots,f_r\rangle$ has no base points away from $x_0$).

\begin{rem}\label{const_general}
We have already seen above that the condition $d\ge rg+r$ ensures that the classes $\wt{\sigma}_{1^r}^g$ and $\wt{\sigma}_{\lambda}$ live in codimension strictly smaller than that of the indeterminacy locus of $\pi$. However, as we will see in Lemma \ref{intersection_generic}, the same condition $d\geq rg+r$ also ensures that $\LL$ contains no points corresponding to degenerate maps $f\colon \bP^1\to\bP^r$. In fact, this is already evident in the case of constant maps; indeed, suppose instead that $d\ge n-1$. Then, we may take the non-zero polynomials $f_0,\ldots,f_r$ to vanish at $x_1,\ldots, x_{n-1}$, and after twisting away all base-points, the resulting map $f\colon \bP^1\to\bP^r$ to be the constant map with image $y_n$. Then $f=[f_0,\ldots, f_r]$ lies on the one hand in $\LL$, and on the other hand in the indeterminacy locus of $\pi$.
\end{rem}

\begin{lem}\label{intersection_generic}
The intersection points of $\wt{\Sigma}_{\lambda}(x_0)$ with $\LL$ are generic in the previous sense. In particular, the intersection occurs away from the indeterminacy locus of $\pi$.
\end{lem}

\begin{proof}
We construct the locus $\LL$ ``relatively,'' allowing the points $y_1,\ldots,y_n$ to vary, and show that, for dimension reasons, the locus where $\LL\cap \wt{\Sigma}_{\lambda}(x_0)$ contains non-generic points cannot dominate the space of choices of the $y_i$. In particular, if the $y_i$ are chosen to be general, we obtain the desired conclusion.

More precisely, let $V\subseteq(\bP^r)^n$ be the open subset of collections of points $y_1,\ldots,y_n\in\bP^r$ where the $y_i$ are in linearly general position, that is, no $m$ of the $y_i$ lie on a linear space of dimension $m-2$ if $2\le m\le r+1$. Consider the product $V\times\bP^{(r+1)(d+1)-1}$, where the second factor parametrizes maps $f=[f_0,\ldots, f_r]$ as before, and the closed subscheme $V\times\wt{\Sigma}_{\lambda}(x_0)$ of the expected dimension $rg$ as defined above. We then define the locus, abusively denoted $\LL\subseteq V\times\bP^{(r+1)(d+1)-1}$, of maps $f$ satisfying $f(x_i)=y_i$ for $i=1,2,\ldots,n$, by relativizing the above construction.

We have a forgetful map $\psi\colon \LL\cap\wt{\Sigma}_{\lambda}(x_0) \to V$, and wish to show that the locus of non-generic points of source does not dominate $V$; to do so, we show that the locus of non-generic points has dimension strictly less than that of $V$.

\vskip 4pt

First, consider the locus on $\LL\cap\wt{\Sigma}_{\lambda}(x_0)$ of non-generic $f=[f_0, \ldots, f_r]\in \LL\cap\wt{\Sigma}_{\lambda}(x_0)$ away from the indeterminacy locus of $\pi$. Suppose that $f$ has base-points of total order $k$ away from $x_0,\ldots, x_n$ and order $k'$ on $x_1,\ldots,x_n$, and that $k+k'>0$. We see upon twisting down by these base-points that the locus of such $f$ has the expected codimension $(r+1)(k+k')$ in $\wt{\Sigma}_{\lambda}(x_0)$, and the incidence conditions $f(x_i)=y_i$ impose at least $(n-k')r$ additional conditions inside $V\times \wt{\Sigma}_{\lambda}(x_0)$. In total, we find that the locus of possible $f$ has codimension strictly greater than $rg+rn$ in $V\times \wt{\Sigma}_{\lambda}(x_0)$, and therefore cannot dominate $V$. Similarly, a parameter count shows that $f$ cannot have ramification sequence strictly more than $\overline{\lambda}$ at $x_0$.

\vskip 4pt


Consider now a point of $\LL\cap\wt{\Sigma}_{\lambda}(x_0)$ for which $\mbox{dim} \langle f_0, \ldots, f_r\rangle\leq r$. We show again by counting parameters that no such $f$ can exist. By twisting away base points at $x_0$ (which decreases the number of moduli and the number of conditions by the same amount), we may assume that $f$ is base point free at $x_0$.  We may also assume that $f$ has no base points away from $x_1,\ldots, x_n$. Suppose now that $f$ has $k$ (simple) base points among these $x_i$, we label them as $x_{n-k+1}, \ldots, x_n$; we twist down our linear series to have degree $d-k$, and lose the corresponding $k$ linear conditions.  Note that in this case the ramification condition at $x_0$ can no longer be imposed in terms of $f_0, \ldots, f_r$ alone, since by assumption, the resulting map $f\colon \bP^1\to\bP^r$ is degenerate, that is, the corresponding linear series has dimension $r'<r$. Note, however, that if the remaining $y_i$ do not themselves live in a linear subspace of $\bP^r$ of dimension $r'$, then this is impossible; we therefore need $n-k\le r'+1$.

Then, it must be true that if $x_1,\ldots,x_{n-k}$ are general points of $\bP^1$, there exists a map $f\colon \bP^1\to\bP^{r'}$ of degree $d-k$ with $f(x_i)=y_i$ for $i=1,\ldots, n-k$. Therefore, we have
\begin{equation*}
(d-k+1)(r'+1)-1\ge r'(n-k).
\end{equation*}
Rearranging yields
\begin{equation*}
k\le d-r'(n-d-1).
\end{equation*}
On the other hand, because $n-k\le r'+1$, we find
\begin{equation*}
(d-n+1)\ge r'(n-d-2).
\end{equation*}
However, by assumption, we have $n\ge d+2$ and $r'\ge0$, so we have reached a contradiction.
\end{proof}

\begin{lem}\label{intersection_transverse}
For a general choice of the points $x_1, \ldots, x_n\in \bP^1$ and $y_1, \ldots, y_n\in \bP^r$, the intersection of $\wt{\Sigma}_{\lambda}(x_0)$ and $\LL$ is transverse.
\end{lem}

\begin{proof}
Let $\cM_{n,d,r}$ be the open subscheme of the space $\Hom_d(\bP^1,\bP^r)\times(\bP^1)^n$ parametrizing elements $([f\colon \bP^1\to\bP^r],x_1,\ldots,x_n)$, where $f$ is a non-degenerate morphism of degree $d$ and the $x_i$ are pairwise distinct points that in addition are distinct from a fixed point $x_0\in\bP^1$.

One may construct $\cM_{n,d,r}$ as an open subset of a $(\bP^1)^{n}$-bundle over $\bP^{(r+1)(d+1)-1}$. We have a smooth, regular map $\chi \colon \cM_{n,d,r}\to\Gr(r+1,d+1)$, from which we can pull back the smooth, open Schubert cycle of linear series with ramification exactly $\overline{\lambda}$ at $x_0$ to obtain the smooth subscheme $Y_{n,d,r}$ parametrizing the morphisms we wish to count. Finally, the projection $\phi\colon Y_{n,d,r}\to(\bP^1)^{n+1}\times(\bP^r)^n$ remembering the marked points and their images on the source is generically unramified of finite degree.

By construction, any non-zero tangent vector to the intersection $\wt{\Sigma}_{\lambda}(x_0)$ and $\LL$ in the generic locus yields a non-zero relative tangent vector of $\phi$. Thus, when the points $x_i, y_i$ are general, there are no such tangent vectors, and the intersection is transverse.
\end{proof}

\vskip 4pt

We are now in a position to complete the proof of Theorem \ref{thm_d_large}.

\begin{proof}[Proof of Theorem \ref{thm_d_large}]
By the above discussion summarized in Proposition \ref{degen_formula}, it suffices to intersect $nr$ linear conditions with $\wt{\sigma}_{1^r}^{g}$ on $\bP^{(r+1)(d+1)-1}$ and compute the degree, that is

$$L_{g,r,d}=\sum_{|\lambda|=rg} \beta_{\lambda} \LL\cdot \widetilde{\sigma}_{\lambda}=\LL\cdot \widetilde{\sigma}_{1^r}^g=\mathrm{deg}\bigl(\widetilde{\sigma}_{1^r}^g\bigr)=\mathrm{deg}\bigl(\sigma_{1^r}^g\bigr),$$
where the last two degree are computed on $\Gr(r+1,d+1)$ and on $\bP^{(r+1)(d+1)-1}$ respectively.

Theorem \ref{thm_d_large} then follows from the fact that the degree of $\wt{\sigma}_{1^r}$ is $r+1$. To see this, note that on $\Gr(r+1,d+1)$, the Schubert cycle $\Sigma_{1^r}$ is the locus of $(r+1)$-planes intersecting a fixed codimension 2 subspace $P\subseteq H^0(\bP^1,\cO(d))$ in a subspace of dimension at least $r$. Identifying $\bP^{(r+1)(d+1)-1}$ with the space of $(r+1)\times(d+1)$ matrices, whose entries are taken up to simultaneous scaling, the pullback of $\Sigma_{1^r}$ may be identified with the determinantal locus of matrices such that the $(r+1)\times 2$ submatrix formed by the first two columns has rank 1. This, in turn, is the pullback under linear projection from $\bP^{(r+1)(d+1)-1}$ of the Segre embedding $\bP^{1}\times \bP^{r}\to\bP^{2r+1}$. Denoting by $h_1$ and $h_2$ the pullbacks to $\bP^1\times \bP^{r+1}$ of the hyperplane classes of $\bP^1$ and $\bP^{r+1}$, observe that
$$\mbox{deg}\bigl(\bP^1\times \bP^{r}\bigr)=(h_1+h_2)^{r+1}={r+1\choose 1}h_1 h_2^r=r+1.$$  This completes the proof.
\end{proof}

\begin{rem}\label{ineq_sharp}
The inequality $d\ge rg+r$ in Theorem \ref{thm_d_large} is sharp. Indeed, the largest possible value of $d$ outside of this range is $d=rg$, corresponding to $d=n-1$. In this case, following the proof of Theorem \ref{thm_d_large} shows that our intersection of cycles inside $\bP^{(r+1)(d+1)-1}$ contains an additional zero-dimensional locus of constant maps $[f_0, \ldots, f_r]$ where each $f_i$ is a constant multiple of the degree $d=n-1$ polynomial vanishing at all of the points $x_1,\ldots,x_n$ except one, $x_i$, and the image of $f$ is the point $y_i$. There is one such map for each of the marked points $x_i$, so we find that $L_{g,r,d}=(r+1)^g-n=(r+1)^g-(d+1)$.
\end{rem}

\begin{rem}It is interesting to observe that on a smooth curve $C$ of genus $g$, a general stable vector bundle $E$ of rank $r+1$ and degree $d$ has precisely $(r+1)^g$ line subbundles of maximal degree $d'$, where $d-(r+1)d'=r(g-1)$, see \cite{OT} or \cite{Ox}. The reinterpretation of the numbers $L_{g,r,d}$ from this point of view will be pursued elsewhere.
\end{rem}

%
%

\subsection{Proof of Theorem \ref{thm_r_or_g_small}}\label{general_computation}

We recast the calculation of the previous section in the following light: we consider the incidence correspondence on $\bP^{(r+1)(d+1)-1}\times\Gr(r+1,d+1)$ of $(r+1)$-tuples of degree $d$ polynomials, spanning a $r+1$-dimensional subspace of
$H^0(\bP^1, \cO_{\bP^1}(d))$, then pull back Schubert cycle conditions on the Grassmannian side and linear conditions on the projective space side. This incidence correspondence is defined by pulling back the diagonal under the map

$$(\pi,\id)\colon \bP^{(r+1)(d+1)-1}\times\Gr(r+1,d+1)\dashrightarrow\Gr(r+1,d+1)\times\Gr(r+1,d+1),$$ and the condition $d\ge rg+r$ is needed in order to prevent the indeterminacy locus from being too large. In this section, we obtain formulas for $L_{g,r,d}$ in the cases $r=1$ and $g\ge1$ by shrinking this base locus.

\vskip 3pt

More precisely, for $j=0,1,\ldots,r$, let $\rho_{j}\colon \bP^{(r+1)(d+1)-1}\dashrightarrow\bP^{d}$ be the linear projection remembering $f_j\in H^0(\cO_{\bP^1}(d))$, where we recall that $(f_0,\ldots,f_r)$ is the $(r+1)$-tuple of polynomials whose coefficients are parametrized by $\bP^{(r+1)(d+1)-1}$. We now consider the following incidence correspondence:
$$\xymatrix{
  & Z:=\Bigl\{\bigl([u], \Lambda\bigr)\in \bP^d\times\Gr(r+1,d+1): u\in \Lambda\Bigr\} \ar[dl]_{\pi_1} \ar[dr]^{\pi_2} & \\
   \bP^d & & \Gr(r+1,d+1)       \\
                 }$$
If $\cQ$ denotes the rank $d-r$ tautological quotient bundle on $\Gr(r+1,d+1)$, then $Z$ can be realized as the degeneracy locus of the composition
\begin{equation*}
\pi_1^*\bigl(\cO_{\bP^d}(-1)\bigr) \longrightarrow \cO_{\bP^d\times \Gr(r+1,d+1)}^{d+1}\longrightarrow \pi_2^*\bigl(\cQ\bigr),
\end{equation*}
and thus has class
\begin{equation*}
\Bigl\{c\bigl(\pi_2^*\cQ\bigr)\cdot c\bigl(\pi_1^*\cO_{\bP^d}(1)\bigr)\Bigr\}_{d-r}=\sum_{i+j=d-r}\pi_2^*\bigl(\sigma_i\bigr)\cdot \pi_1^*\bigl(H^j\bigr)\in \CH^{d-r}\bigl(\bP^d\times\Gr(r+1,d+1)\bigr),
\end{equation*}
where $H$ is the hyperplane class on $\bP^d$ and where we have also used that $c_j(\cQ)=\sigma_j$.

Because the codimension of the base locus of $\rho_j$ is $d+1>d-r$, the closure
$$Z_j:=\bigl(\rho_j\times \mathrm{id}_{\Gr(r+1,d+1)}\bigr)^{-1}(Z)$$ of the pullback of the correspondence $Z$ has the same class, that is,
$\sum_{i+j=d-r}  \pi_2^*\bigl(\sigma_i\bigr)\cdot \pi_1^*\bigl(H^j\bigr)$, where this time
$\pi_2\colon \bP^{(r+1)(d+1)-1}\times \Gr(r+1,d+1)\rightarrow \Gr(r+1,d+1)$ denotes the second projection.

\vskip 5pt

\begin{proof}[Proof of Theorem \ref{thm_r_or_g_small}]
We wish to compute the intersection inside $\bP^{(r+1)(d+1)-1}\times \Gr(r+1,d+1)$ of the $nr$ linear conditions pulled back from $\bP^{(r+1)(d+1)-1}$ given by the equations $f(x_i)=y_i$ for $i=1, \ldots, n$,  the pullback under $\pi_2$ of the Schubert cycles $\sigma_\lambda$, where $\lambda$ is a Schubert index with $|\lambda|=rg$, and the classes of the cycles $Z_0,\ldots, Z_r$ defined above. We proceed as in Lemmas \ref{intersection_generic} and \ref{intersection_transverse}.

\vskip 4pt

First we introduce the incidence correspondence:
\begin{equation*}
\xymatrix
{
  & \cX :=\left\{\left([f_0, \ldots, f_r], \Lambda\right)\in \bP^{(r+1)(d+1)-1}\times\Gr(r+1,d+1): f_i\in \Lambda \right\} \ar[dl]_{\pi_1} \ar[dr]^{\pi_2} & \\
   \bP^{(r+1)(d+1)-1} & & \Gr(r+1,d+1)
                 }
\end{equation*}

We first claim that the intersection $\pi_2^*(\sigma_{\lambda})\cdot \pi_1^*(\LL)$ in question is supported away from the locus of $(f,\Lambda)\in \cX$, where $f$ defines a degenerate map $f\colon \bP^{1}\to\bP^{r'}$, for some $r'<r$. When $d\ge rg+r$, the same proof as in Lemma \ref{intersection_generic} applies.

Suppose that either $r=1$ or $n=d+2$, there is such a $(f,\Lambda)\in Z$ in our intersection, and that $k$ of the points $x_1,\ldots,x_n$ are base-points of the $r'$-dimensional linear system $\Lambda_f$ spanned by $f_0,\ldots,f_r$. As in the proof of Lemma \ref{intersection_generic}, it must be the case that $k\ge n-r'-1$.

\vskip 4pt

Denote the total ramification of $\Lambda_f$ at $x_0$ by $t$. Then,
\begin{align*}
t&\le \dim\Gr(r'+1,d-k+1)\\
&\le \dim\Gr(r'+1,d-n+r'+2)\\
&=(r'+1)(d-n+1).
\end{align*}

Thus,
\begin{align*}
d&\ge \frac{t}{r'+1}+n-1\\
&=\frac{t}{r'+1}+d+\frac{d}{r}-g,
\end{align*}
whence
\begin{equation*}
d\le rg-\frac{rt}{r'+1}.
\end{equation*}

On the other hand, we require that $\Lambda\in\sigma_\lambda$, where $|\lambda|=rg$, and $\Lambda_f\subseteq  \Lambda$. Such a $\Lambda$ can only exist if
\begin{equation*}
t+(r-r')(d-r)\ge rg,
\end{equation*}
as $\dim(\Lambda/\Lambda_f)=r-r'$, and each dimension can contribute at most $d-r$ to the ramification of $\Lambda$ at $x_0$. Because $r>r'$, we obtain
\begin{align*}
t+(r-r')\left(rg-\frac{rt}{r'+1}-r\right)&\ge rg\\
t\left(1-\frac{(r-r')r}{r'+1}\right)+(r-r')(rg-r)&\ge rg.
\end{align*}

When $r=1$, and thus $r'=0$, we obtain a contradiction. It remains to consider the case $n=r+2$, in which case
\begin{equation*}
rg=(d-r)(r+1).
\end{equation*}
Then, comparing the inequalities
\begin{align*}
t&\le (r'+1)(d-n+1)=(r'+1)(d-r-1)\\
t&\ge rg-(r-r')(d-r)=(d-r)(r+1)-(r-r')(d-r)
\end{align*}
also yields a contradiction.

\vskip 4pt

Therefore, we are back in the situation of Lemma \ref{intersection_generic} in which all intersection points occur where $f$ is non-degenerate, and in particular, $(f,\Lambda)$ lies away from the indeterminacy of the $\rho_j$. The same parameter counts show that $f$ indeed defines a map $f\colon \bP^1\to\bP^1$ of degree $d$ with vanishing exactly $\overline{\lambda}$ at $x_0$.

\vskip 4pt

Furthermore, the intersection in question is transverse by exactly the same argument as in Lemma \ref{intersection_transverse}, so it suffices to compute the degree of the intersection cycle on $\bP^{(r+1)(d+1)-1}\times \Gr(r+1,d+1)$.  This equals
$$
\int_{\bP^{(r+1)(d+1)-1}\times\Gr(r+1,d+1)}\pi_2^*\bigl(\sigma_{1^r}^g\bigr) \cdot \pi_1^*\bigl(H^{nr}\bigr) \left(\sum_{i+j=d-r}\pi_2^*\bigl(\sigma_i\bigr)\cdot \pi_1^*\bigl(H^j\bigr)\right)^{r+1}$$
$$=\int_{\Gr(r+1,d+1)}\sigma_{1^r}^g\cdot\left[ \sum_{\alpha_0+\cdots+\alpha_{r}=(r+1)(d-r)-rg}\left(\prod_{i=0}^{r}\sigma_{\alpha_i}\right)\right],
$$
as desired.
\end{proof}

\begin{rem}\label{const_higherrank}
While the proof of Theorem \ref{thm_r_or_g_small} shows that our refined incidence correspondence avoids the constant maps of Remark \ref{const_general} when $r=1$ or $n=r+2$, this is not the case in general. Indeed, suppose that $r\ge2$, $n\ge r+3$, and $d\ge n-1$. Then, take $f_0,\ldots,f_r$ to have simple zeroes at $x_1,\ldots,x_{n-1}$ and an order $d-(n-1)$ zero at $x_0$, such that $f=[f_0, \ldots, f_r]$ defines the constant map with image $y_n\in\bP^r$. If $\lambda$ is a Schubert index with $|\lambda|=rg$, then the condition that $f_0,\ldots,f_r\in\Lambda$, where $\Lambda\in\Sigma_\lambda(x_0)\subseteq \Gr(r+1,d+1)$, may be satisfied as long as $rg\le (d-n+1)+\dim \Gr(r,d)=(d-n+1)+r(d-r)$. Substituting $rg=dr+r+d-rn$, this is equivalent to $n\ge\frac{r^2+r-1}{r-1}=r+2+\frac{1}{r-1}$. When $r\ge2$ and $n\ge r+3$, this is immediate.
\end{rem}

For $r=1$, Theorem \ref{thm_r_or_g_small} can be used to recover (via simple manipulations) the explicit formulas in terms of binomial coefficients for the degrees $L_{g,1,d}$. These numbers are also determined in \cite{cps} using excess intersection on Hurwitz spaces of admissible covers.

\begin{prop}\label{prop:explicit}
For $d\geq \frac{g+2}{2}$, we have
\begin{align*}
L_{g,1,d}&=\sum_{\alpha_0+\alpha_1=2d-2-g}\int_{\Gr(2,d+1)} \sigma_1^g\cdot \sigma_{\alpha_0}\cdot \sigma_{\alpha_1}\\
&=
\sum_{i=0}^{\bigl \lfloor \frac{2d-g-2}{2}\bigr \rfloor} \frac{(2d-g-2i-1)^2}{g+1} {g+1\choose d-i}\\
&=2^g-2\sum_{i=0}^{g-d-1}\binom{g}{i}+(g-d-1)\binom{g}{g-d}+(d-g-1)\binom{g}{g-d+1}.
\end{align*}
where in the last line, we take $\binom{g}{j}=0$ when $j<0$.
\end{prop}
\begin{proof} We use Giambelli's formula $\sigma_{a,b}=\sigma_a\cdot \sigma_b-\sigma_{a+1}\cdot\sigma_{b-1}\in \CH^{a+b}\bigl(\Gr(2,g+1)\bigr)$ for $a\geq b$, coupled with the formula, see e.g. \cite[p. 269]{GH80}
\begin{align*}
\int_{\Gr(2,d+1)}\sigma_{a,b}\cdot \sigma_1^g&=\frac{a-b+1}{g+1}\cdot\binom{g+1}{d-b}\\
&=\binom{g}{d-b-1}-\binom{g}{d-b}
\end{align*}
for all $a\geq b$ with $a+b=2d-2-g$. Substituting in the formula provided by Theorem \ref{thm_r_or_g_small} yields the claims.
\end{proof}

\subsection{Degrees of determinantal Schubert cycles}

We note here that comparison of the incidence correspondences given above in the proofs of Theorem \ref{thm_d_large} and \ref{thm_r_or_g_small} allows one to compute the degrees of pullbacks of Schubert cycles of low codimension on $\Gr(r+1,d+1)$ to $\bP^{(r+1)(d+1)-1}$.

\begin{prop}\label{schubert_degree}
Let $\Sigma_{\lambda}$ be a Schubert cycle of codimension $|\lambda|\le d-r$ in $\Gr(r+1,d+1)$, and let $\wt{\Sigma}_{\lambda}$ be the closure of its pullback under the rational map $\pi\colon \bP^{(r+1)(d+1)-1}\dashrightarrow \Gr(r+1,d+1)$. Then, the degree of $\wt{\Sigma}_{\lambda}$ is
\begin{equation*}
    \int_{\Gr(r+1,d+1)}\sigma_\lambda\cdot\left[ \sum_{\alpha_0+\cdots+\alpha_{r}-|\lambda|=(r+1)(d-r)-rg}\left(\prod_{i=0}^{r}\sigma_{\alpha_i}\right)\right].
\end{equation*}
\end{prop}

\begin{proof}
Let $N=(r+1)(d+1)-1-|\lambda|$. Recall that the codimension of $\Sigma_{\lambda}$ is strictly smaller that the codimension of the indeterminacy locus of $\pi$. Accordingly, adopting the notation of the previous two sections, we have
\begin{align*}
\deg(\wt{\Sigma}_{\lambda})&=\int_{\bP^{(r+1)(d+1)-1}}H^{N}\cdot[\wt{\Sigma}_{\lambda}]\\
&=\int_{\bP^{(r+1)(d+1)-1}\times\Gr(r+1,d+1)}\pi_{1}^{*}\bigl(H^{N}\cdot[\wt{\Sigma}_{\lambda}]\bigr)\cdot[Z_0]\cdots[Z_r]\\
&=\int_{\bP^{(r+1)(d+1)-1}\times\Gr(r+1,d+1)}\pi_{1}^{*}\bigl(H^{N}\bigr)\cdot\pi_2^{*}(\Sigma_{\lambda})\cdot[Z_0]\cdots[Z_r]\\
&=\int_{\bP^{(r+1)(d+1)-1}\times\Gr(r+1,d+1)}\pi_{1}^{*}\bigl(H^{N}\bigr)\cdot\pi_2^{*}(\Sigma_{\lambda})\cdot\left(\sum_{i+j=d-r}\pi_2^*\bigl(\sigma_i\bigr)\cdot \pi_1^*\bigl(H^j\bigr)\right)^{r+1}\\
&=\int_{\Gr(r+1,d+1)}\sigma_\lambda\cdot\left[ \sum_{\alpha_0+\cdots+\alpha_{r}=(r+1)(d-r)-rg-|\lambda|}\left(\prod_{i=0}^{r}\sigma_{\alpha_i}\right)\right],
\end{align*}
where we have used the equality
\begin{equation*}
\pi_{1}^{*}(\wt{\Sigma}_{\lambda})\cap Z_0\cap\cdots\cap Z_r=\pi_{2}^{*}(\Sigma_{\lambda})\cap Z_0\cap\cdots\cap Z_r
\end{equation*}
as \textit{subschemes} of the incidence correspondence $\cX$.
\end{proof}

\section{Young Tableaux Interpretation}\label{tableaux}

Comparison of Theorems \ref{thm_d_large} and \ref{thm_r_or_g_small} yields the following purely combinatorial statement.

\begin{prop}\label{comb_identity}
Suppose that $g\ge0,r\ge1,d\ge rg+r$, and $d$ is divisible by $r$. Then,
\begin{equation*}
    \int_{\Gr(r+1,d+1)}\sigma_{1^r}^g\cdot\left[ \sum_{\alpha_0+\cdots+\alpha_{r}=(r+1)(d-r)-rg}\left(\prod_{i=0}^{r}\sigma_{\alpha_i}\right)\right]=(r+1)^g.
\end{equation*}
\end{prop}

Indeed, both sides are equal to $L_{g,d,r}$ whenever $n=d-g+1+\frac{d}{r}$ is an integer. However, when $d\ge g+r$, both sides are independent of $d$; for the left hand side, this can be seen in terms of Schubert calculus, but will also be made transparent in the combinatorial interpretation that follows. In particular, Proposition \ref{comb_identity} holds under the weaker inequality $d\ge g+r$ with no condition on the divisibility by $r$.

We give a combinatorial interpretation of the left hand side in terms of Young Tableaux. Consider a filling of the boxes of a $(r+1)\times(d-r)$ grid with:
\begin{itemize}
    \item $rg$ red integers among \textcolor{red}{$1,2,\ldots,g$}, with each appearing exactly $r$ times, and
    \item $(r+1)(d-r)-rg$ blue integers among \textcolor{blue}{$0,1,\ldots,r$}, with each appearing any number of times,
\end{itemize}
subject to the following conditions:
\begin{itemize}
    \item the red integers are top- and left- justified, i.e., they appear above blue integers in the same column and to the left of blue integers in the same row,
    \item the red integers are strictly increasing across rows and weakly increasing down columns
    \item the blue integers are weakly increasing across rows and strictly increasing down columns.
\end{itemize}
An example filling is given in the case $(g,d,r)=(6,15,2)$ below.
\begin{center}
\begin{tabular}{ |c|c|c|c|c|c|c|c|c|c|c|c|c|}
\hline
\textcolor{red}{1} & \textcolor{red}{2} & \textcolor{red}{3} & \textcolor{red}{4} & \textcolor{red}{6} & \textcolor{blue}{0} & \textcolor{blue}{0}  &\textcolor{blue}{0} & \textcolor{blue}{0} & \textcolor{blue}{0} & \textcolor{blue}{0} & \textcolor{blue}{0} & \textcolor{blue}{0} \\
\hline
\textcolor{red}{1} & \textcolor{red}{3} & \textcolor{red}{5} & \textcolor{red}{6} &  \textcolor{blue}{0} &\textcolor{blue}{1} & \textcolor{blue}{1} & \textcolor{blue}{1} & \textcolor{blue}{1} & \textcolor{blue}{1} & \textcolor{blue}{1} & \textcolor{blue}{1} & \textcolor{blue}{1}\\
\hline
\textcolor{red}{2} & \textcolor{red}{4} & \textcolor{red}{5} & \textcolor{blue}{0} & \textcolor{blue}{2} & \textcolor{blue}{2} & \textcolor{blue}{2} & \textcolor{blue}{2} & \textcolor{blue}{2} & \textcolor{blue}{2} & \textcolor{blue}{2} & \textcolor{blue}{2} & \textcolor{blue}{2}\\
\hline
\end{tabular}
\end{center}

Note that the rightmost $d-r-g$ columns must be filled with the blue integers $0,1,\ldots,r$ in order, so a filling as above is determined by the leftmost $g$ columns, which are those that may contain red integers. In particular, the number of such fillings is independent of $d$ when $d\ge r+g$. Now, we claim that this number of fillings is given exactly by the intersection number on the left hand side of Proposition \ref{comb_identity}. Indeed, by the Pieri rule, the term $\sigma_{1^r}$ corresponds to the transposed semi-standard Young Tableau given by the red integers, and the broken strips formed by the blue entries equal to $i$ correspond to the Schubert cycle $\sigma_{\alpha_i}$.

Proposition \ref{comb_identity} therefore implies:

\begin{prop}\label{count_fillings}
Suppose $d\ge r+g$. Then, the number of fillings of a $(r+1)\times(d-r)$ grid satisfying the above conditions is equal to $(r+1)^g$.
\end{prop}

A combinatorial proof of Proposition \ref{count_fillings} via the RSK algorithm has been given by Gillespie-Reimer-Berg \cite{grb}.

\section{Variants}

\subsection{Linear series with fixed incidences and secancy conditions}\label{tev_computation}

We briefly explain how our methods also recover the more general \textit{Tevelev degrees} of \cite{cps}, where some of the points $x_i$ are constrained to lie in the same fiber of $f$. Recall from \S\ref{intro} that, if $1\le k\le n$, we defined $L'_{g,d,k}$ to be the number of morphisms $f\colon C\to\bP^1$ of degree $d$ sending general points $x_1,\ldots,x_n\in C$ to points $y_1,\ldots,y_n\in\bP^1$, where $y_1=y_2=\cdots=y_k$ but the $y_i$ are otherwise general.

\vskip 4pt

More generally, we may fix integers $0\leq a\leq k\leq d$,  a general $n$-pointed curve $(C, x_1, \ldots, x_k, x_{k+1},\ldots, x_n)$ of genus $g$, where $n$ is given by (\ref{eq:n}), and consider the variety
$$G^{r,k-a}_{d,k}(C,x_1, \ldots, x_k):=\Bigl\{\ell\in G^r_d(C): \mathrm{dim }\ \ell(-x_1-\cdots-x_{k})\geq r-k+a \Bigr\},$$
parametrizing linear systems $\ell$ whose induced map $\phi_{\ell} \colon C\dashrightarrow \bP^r$ has the property that
$$\bigl\langle \phi_{\ell}(x_1), \ldots, \phi_{\ell}(x_k)\bigr\rangle \cong \bP^{k-a-1}.$$
Then $G^{r,k-a}_{d,k}(C,x_1, \ldots, x_k)$ is a determinantal variety of dimension
$$\rho(g,r,d)-a(r+1-k+a).$$

Fixing points $y_1, \ldots, y_n\in \bP^r$ general with the property that $$\mbox{dim } \langle y_1, \ldots, y_k\rangle =k-1-a,$$ one can ask for the number of maps $f\colon C\rightarrow \bP^r$ of degree $d$ such that $f(x_i)=y_i$ for $i=1,\ldots, n$. For any such map, the corresponding linear series $\ell:=f^*\bigl|\cO_{\bP^r}(1)\bigr|$ lies in $G^{r,k-a}_{d,k}(C,x_1, \ldots, x_k)$.

In the interest of simplicity we deal only with the case
$$r=1, \ a=k-1,$$
in which case this number equals $L_{g,d,k}'$. We only sketch the proof; we refer the reader to \cite[\S 6]{celalian} for detailed proofs and more general statements.

\vskip 4pt

\begin{proof}[Proof of Theorem \ref{thm_cps}]
Consider a linear series $V$ on our general curve $C$ satisfying the needed incidence conditions. We employ a further degeneration after that of \S\ref{degeneration}, allowing $x_1,\ldots,x_k$ to coalesce onto a bubbled rational component $R_k$, attached to $R_0$ at $x$, and consider the resulting limit $V_0$ on this bubbled curve. \footnote{As explained in \cite[\S 6]{celalian}, one should more precisely consider the degeneration of the data of both $V$ and two (possibly linear dependent) sections of $V$ defining a map $f:C\to\bP^1$. We do not discuss the details here.} We find that the $R_k$-aspect of $V_0$ must have ramification sequence $(d-k,d-1)$ at $x$, and sends $x_1,\ldots,x_k$ to the same point after twisting down the base-points at $x$.

It now suffices to count linear series on $R_0$ with the aggregate ramification condition $\sigma_1^g$ at $x_0$, the new ramification condition $\sigma_{k-1}$ at $x$, an additional linear incidence condition at $x$ (with image $y_1=\cdots=y_k$), and linear incidence conditions at $x_{k+1},\ldots,x_n$. The computation of \S\ref{general_computation} yields the count
\begin{equation*}
\int_{\Gr(2,d+1)}\sigma_1^g\sigma_{k-1}\cdot\left[\sum_{i+j=2(d-1)-g-(k-1)}\sigma_{i}\sigma_j\right].
\end{equation*}

However, we find the following extraneous solutions: if the linear series in question has a base-point at $x$, then we twist down, so that the new ramification sequence is $(0,k-2)$, and $d$ decreases by 1; in addition, we lose the linear incidence condition at $x$. Therefore, we see a (zero-dimensional) excess contribution of
\begin{equation*}
\int_{\Gr(2,d)}\sigma_1^g\sigma_{k-2}\cdot\left[\sum_{i+j=2(d-2)-g-(k-2)}\sigma_{i}\sigma_j\right].
\end{equation*}
Subtracting the above yields the formula for $L'_{g,d,k}$. One needs to check that there are no additional degenerate contributions and that the intersections are transverse as before, but we omit the details.
\end{proof}

Applying the Pieri rule to the formula of Theorem \ref{thm_cps} yields the following recursions, recovering \cite[Proposition 7]{cps} after the change of coordinates $\Tev_{g,\ell,r}=L'_{g,g+\ell+1,r}$. These recursions are then used in \cite{cps} to obtain explicit formulas in terms of binomial coefficients.

\begin{cor}\label{recursion}
We have:
\begin{equation*}
    L'_{g,d,1}=L'_{g-1,d-1,1}+L'_{g-1,d,2}
\end{equation*}
and
\begin{equation*}
    L'_{g,d,k}=L'_{g-1,d-1,k-1}+L'_{g-1,d,k+1}
\end{equation*}
for $k>1$.
\end{cor}

\begin{rem}\label{cps_const}
The proof of Theorem \ref{thm_d_large} may also be employed to show that $L'_{g,d,k}=2^g$ whenever $n\ge d+k+1$. For general $r$, the number of linear series in question is $(r+1)^g$ whenever $n\ge d+a+2$.

However, even when $r=1$, the proof of Theorem \ref{thm_r_or_g_small} breaks down as soon as $k>1$, as we will see contributions from constant maps with value $y_1=\cdots=y_k$ and base-points at $x_{k+1}=\cdots=x_n$. Thus, the additional degeneration as above is needed to obtain the general formula for $L'_{g,d,k}$.
\end{rem}

\subsection{Linear series with imposed incidences and prescribed ramification}\label{ram_conditions}

We fix a general pointed curve $[C, p_1,\ldots, p_m, x_1,\ldots,x_n]\in\cM_{g,m+n}$, general points  $y_1,\ldots,y_n\in\bP^r$, as well as $m$ partitions $\lambda_1, \ldots, \lambda_m$ of length $r+1$. We may consider morphisms $f\colon C\to\bP^r$ of degree $d$ satisfying $f(x_i)=y_i$ for $i=1, \ldots, n$ and $f$ has ramification sequence at least $\overline{\lambda_j}$ at $p_j$ for $j=1, \ldots, m$. Suppose for simplicity that the $(r+1)$-st part of each $\lambda_j$ is zero, so that $f$ has no base points. Equivalently, like in (\ref{eq:evaluation}) we can consider the evaluation map
\begin{equation}\label{eq:ramif}
\mathrm{ev}_{(x_1, \ldots, x_n)}\colon G^r_d\Bigl(C, \bigl(p_1, \overline{\lambda}_1\bigr), \ldots, \bigl(p_m, \overline{\lambda}_m)\Bigr)\dashrightarrow P_r^n,
\end{equation}
and ask for its degree when the dimension of the two varieties in question are equal. Using (\ref{adjBN}), we expect a finite number of such maps $f\colon C\rightarrow \bP^r$ whenever
$\rho(g,r,d, \overline{\lambda}_1, \ldots, \overline{\lambda}_m)=rn-(r^2+2r)$, that is, when
\begin{equation}\label{nramif}
    n=\frac{dr+d+r-\lambda_{\tot}-gr}{r},
\end{equation}
where $\lambda_{\tot}:=|\lambda_1|+\cdots +|\lambda_m|$ is the total size of the partitions $\lambda_j$. Let $L_{g,r,d}^{\lambda_1,\ldots,\lambda_m}$ be this number, that is, the degree of the map given by (\ref{eq:ramif}).

\vskip 5pt

Degenerating the general genus $g$ curve $C$ to a flag curve as in \S\ref{degeneration} so that the points $p_1, \ldots, p_m$ specialize to general points on the component $R_{\spine}$, whereas $x_1, \ldots, x_n$ specialize, as before, to general points of the rational component $R_0$, we reduce the computation to the numbers $L_{g,r,d,\lambda}$, as defined in Definition \ref{ls_count_lambda}, where now $|\lambda|=rg+\lambda_{\tot}$. Following the proof of Theorem \ref{thm_d_large}, we obtain the following result.

\begin{prop}\label{thm_add_ram1}
Suppose that $d\ge rg+r+\lambda_{\tot}$, or equivalently $n\ge d+2$. Then,
\begin{equation*}
L_{g,r,d}^{\lambda_1,\ldots,\lambda_m}=(r+1)^g\cdot \prod_{j=1}^{m}\deg(\wt{\Sigma}_{\lambda_j}),
\end{equation*}
where $\deg(\wt{\Sigma}_{\lambda_j})$ is the degree of the cycle $\wt{\Sigma}_{\lambda_j}$ on $\bP^{(r+1)(d+1)-1}$ obtained by taking the closure of the pullback of $\Sigma_{\lambda_j}(x_0)\subseteq \Gr(r+1,d+1)$ under the rational map $\pi\colon \bP^{(r+1)(d+1)-1}\dashrightarrow \Gr(r+1,d+1)$, see Proposition \ref{schubert_degree}.
\end{prop}

\vskip 3pt

Similarly, closely following the proof of Theorem \ref{thm_r_or_g_small}, we obtain:

\begin{prop}\label{thm_add_ram2}
Suppose that:
\begin{itemize}
\item $d\ge rg+r+\lambda_{\tot}$,
\item $n=r+2$, or
\item $r=1$.
\end{itemize}

Then,
\begin{equation*}
L_{g,r,d}^{\lambda_1,\ldots,\lambda_m}=\int_{\Gr(r+1,d+1)}\sigma_{1^r}^{g}\cdot\prod_{j=1}^{m}\sigma_{\lambda_j}\cdot\left[ \sum_{\alpha_0+\cdots+\alpha_{r}=(r+1)(d-r)-rg-\lambda_{\tot}}\left(\prod_{i=0}^{r}\sigma_{\alpha_i}\right)\right].
\end{equation*}
\end{prop}

Indeed, in both results, the only significant modification is that the total ramification imposed at $x_0$ after degeneration is $rg+\lambda_{\tot}$, instead of $rg$. However, this number is equal to $dr+d+r-nr$ in both cases, and from here the proofs go through without change.

\end{document}